\documentclass[preprint,12pt]{elsarticle}

\usepackage[T1]{fontenc}
\usepackage[utf8]{inputenc}
\usepackage{lmodern}
\usepackage{microtype}
\usepackage{amsmath,amssymb,amsthm,mathtools}
\usepackage{booktabs,array,tabularx}
\usepackage{enumitem}
\usepackage{xcolor}
\usepackage{aliascnt}
\usepackage{tikz}
\usetikzlibrary{calc}
\usepackage{hyperref}
\usepackage[nameinlink,capitalise,noabbrev]{cleveref}

\hypersetup{
  colorlinks=true,
  linkcolor=blue!45!black,
  citecolor=blue!45!black,
  urlcolor=blue!45!black
}

\journal{xxx}

\newtheorem{theorem}{Theorem}[section]

\newaliascnt{lemma}{theorem}
\newtheorem{lemma}[lemma]{Lemma}
\aliascntresetthe{lemma}

\newaliascnt{proposition}{theorem}
\newtheorem{proposition}[proposition]{Proposition}
\aliascntresetthe{proposition}

\newaliascnt{corollary}{theorem}
\newtheorem{corollary}[corollary]{Corollary}
\aliascntresetthe{corollary}

\theoremstyle{definition}
\newaliascnt{definition}{theorem}
\newtheorem{definition}[definition]{Definition}
\aliascntresetthe{definition}

\newaliascnt{problem}{theorem}
\newtheorem{problem}[problem]{Problem}
\aliascntresetthe{problem}

\newaliascnt{example}{theorem}
\newtheorem{example}[example]{Example}
\aliascntresetthe{example}

\theoremstyle{remark}
\newaliascnt{remark}{theorem}
\newtheorem{remark}[remark]{Remark}
\aliascntresetthe{remark}

\crefname{theorem}{Theorem}{Theorems}
\Crefname{theorem}{Theorem}{Theorems}
\crefname{lemma}{Lemma}{Lemmas}
\Crefname{lemma}{Lemma}{Lemmas}
\crefname{proposition}{Proposition}{Propositions}
\Crefname{proposition}{Proposition}{Propositions}
\crefname{corollary}{Corollary}{Corollaries}
\Crefname{corollary}{Corollary}{Corollaries}
\crefname{definition}{Definition}{Definitions}
\Crefname{definition}{Definition}{Definitions}
\crefname{problem}{Problem}{Problems}
\Crefname{problem}{Problem}{Problems}
\crefname{example}{Example}{Examples}
\Crefname{example}{Example}{Examples}
\crefname{remark}{Remark}{Remarks}
\Crefname{remark}{Remark}{Remarks}

\newcommand{\F}{\mathbb F}
\newcommand{\Ftwo}{\mathbb F_2}
\newcommand{\cC}{\mathcal C}
\newcommand{\cF}{\mathcal F}
\newcommand{\cI}{\mathcal I}
\newcommand{\cS}{\mathcal S}
\newcommand{\cT}{\mathcal T}
\newcommand{\cB}{\mathcal B}
\newcommand{\CS}{\operatorname{CS}}

\newcommand{\Ort}{\operatorname{Ort}}

\newcommand{\ind}{\mathbf 1}
\newcommand{\tr}{\mathsf T}
\newcommand{\symdiff}{\mathbin{\triangle}}

\newcolumntype{Y}{>{\raggedright\arraybackslash}X}

\begin{document}

\begin{frontmatter}

\title{A Fano framework for binary delta-matroids}

\renewcommand{\thefootnote}{\fnsymbol{footnote}}

\author{Zhuo Li$^{1}$, Xian'an Jin$^{1,3}$, Qi Yan$^{2}$\footnote{Corresponding author.}\\
        ~\\
        \small $^{1}$School of Mathematical Sciences, Xiamen University, P. R. China\\
		\small $^{2}$School of Mathematics and Statistics, Lanzhou University, P. R. China\\
        \small $^{3}$School of Mathematics and Statistics, Qinghai Minzu University, P. R. China\\
    \small\tt Email:   lzhuo@stu.xmu.edu.cn, yanq@lzu.edu.cn, xajin@xmu.edu.cn}

\begin{abstract}
Dunshee and Ellingham recently showed that seven natural properties
of a cellularly embedded graph form a Fano-plane framework. We
establish an analogous framework for binary delta-matroids. For a
binary delta-matroid $D$ and $\tau$, let
$Z_3(D,\tau)$ denote its associated binary tight $3$-matroid. The
six outer points are represented by evenness or bipartiteness of
$D$ and its global vertex-flip transforms. For the seventh point,
we call $D$ $Z_3$-bipartite when every circuit of $Z_3(D,\tau)$
has even cardinality. We show that the satisfied properties are
precisely the nonzero vectors of a subspace of $\Ftwo^3$. For
ribbon-graphic delta-matroids, $Z_3$-bipartiteness is equivalent to
bipartiteness of the medial graph, so the construction recovers the
Fano-plane framework for embedded graphs.
\end{abstract}


\end{frontmatter}

\section{Introduction}\label{sec:intro}

Delta-matroids were introduced by Bouchet as a common extension of matroids and several set systems arising in graph theory and topological graph theory \cite{Bouchet1987,Bouchet1989}. Their connection with embedded graphs is particularly fruitful. A ribbon graph $G$ has an associated delta-matroid $D(G)$, and the principal topological operations on $G$ become natural operations on $D(G)$: partial duality corresponds to twist and partial Petrie duality corresponds to loop complementation \cite{Chmutov2009,ChunEtAl2019JCTA,ChunEtAl2019PLMS}. 

Dunshee and Ellingham recently discovered a Fano-plane framework among seven properties of a cellularly embedded graph \cite{DunsheeEllingham2025}. They label the properties by the seven nonzero vectors of $\Ftwo^3$ and prove that the satisfied points are the nonzero vectors of a subspace of $\Ftwo^3$.
In their concluding remarks, Dunshee and Ellingham explicitly proposed extending this framework beyond embedded graphs. In paraphrased form, their question is the following.

\begin{problem}\cite{DunsheeEllingham2025}
\label{prob:dunshee-ellingham}
Extend the Fano framework for graph embeddings to binary, or more generally vf-safe, delta-matroids. In particular, identify an appropriate delta-matroidal counterpart of the medial graph and a property corresponding to medial-graph bipartiteness.
\end{problem}

They observed that twisted duality already has a natural delta-matroidal interpretation, but that no evident counterpart of the medial graph was available. We solve Problem~\ref{prob:dunshee-ellingham} for binary delta-matroids and show by example that vf-safety alone does not suffice for the same Fano-line implications.

For a binary delta-matroid $D$ on $E$, write
\[
D^{*}:=D*E,\qquad D^{+}:=D+E,
\]
and compose successive vertex flips from left to right. 
Yan and Jin obtained an important first relation in this direction: for a binary even delta-matroid $D$, bipartiteness of $D$ is equivalent to evenness of $D^+$ \cite{YanJin2022}. This gives one Fano line once one of its points is assumed. Our purpose is to derive all seven lines simultaneously and to identify the missing seventh point.

We consider the following properties:
\begin{equation}\label{eq:seven-intro}
\begin{array}{c|l}
001 & D\text{ is even},\\
010 & D^{+}\text{ is even},\\
011 & D\text{ is bipartite},\\
100 & D^{*+}\text{ is even},\\
101 & D^{*}\text{ is bipartite},\\
110 & D^{+*}\text{ is bipartite},\\
111 & \text{every circuit of }Z_3(D,\tau)\text{ has even cardinality}.
\end{array}
\end{equation}
We call the last property \emph{$Z_3$-bipartiteness}. 

The proof rests on two circuit-parity characterizations. Fix an ordered transversal triple $(T_1,T_2,T_3)$ of $Z_3(D,\tau)$ compatible with $D$, and for a circuit $C$ put
\[
p_i(C)=|C\cap T_i|\pmod 2.
\]
We prove
\[
D\text{ is even}
\quad\Longleftrightarrow\quad
p_3(C)=0\quad\text{for every circuit }C,
\]
and
\[
D\text{ is bipartite}
\quad\Longleftrightarrow\quad
p_2(C)+p_3(C)=0\quad\text{for every circuit }C.
\]
Permuting the three transversals gives the six outer conditions in \eqref{eq:seven-intro}, while $Z_3$-bipartiteness is the equation $p_1+p_2+p_3=0$. Thus the seven points are precisely the seven nonzero linear parity functionals on $(p_1,p_2,p_3)$, and the Fano structure follows immediately from linearity over $\Ftwo$.

The paper is organized as follows. Section~\ref{sec:prelim} fixes
the notation and records the required results on binary matroids,
delta-matroids, multimatroids, and tight
$3$-matroid. Section~\ref{sec:parity} proves the evenness and
bipartiteness parity criteria. Section~\ref{sec:fano} establishes
the Fano framework and its immediate consequences.
Section~\ref{sec:ribbon-specialization} specializes the theory to
ribbon graphs and identifies $Z_3$-bipartiteness with medial-graph
bipartiteness. Section~\ref{sec:vfsafe-counterexample} gives a
counterexample showing that Theorem~\ref{thm:fano-framework} does not hold for vf-safe delta-matroids.

\section{Preliminaries}\label{sec:prelim}
All sets and combinatorial structures in this paper are finite.
We write $\symdiff$ for symmetric difference and work over
$\Ftwo$ whenever a field is not explicitly specified. We use
standard matroid terminology without further comment; see
\cite{Oxley2011, Brijder2018}. The definitions and results on delta-matroids
and multimatroids needed below are recalled explicitly.

\subsection{Binary matroids and cycle spaces}\label{subsec:matroids}

A \emph{set system} is a pair $S=(E,\cF)$, where $E$ is a finite set
and $\cF\subseteq 2^E$. The members of $\cF$ are called
\emph{feasible sets}. A set system is called \emph{proper} if
$\cF\neq\varnothing$.
For $X\subseteq E$, define
\[
\cF|X:=\{F\in\cF:F\subseteq X\}, \qquad
S|X:=(X,\cF|X)
\]
and call $S|X$ the \emph{restriction} of $S$ to $X$.

A \emph{matroid} is a proper set system $M=(E,\cB)$ satisfying the
basis exchange axiom: for all $B_1,B_2\in\cB$ and every
$e\in B_1\setminus B_2$, there exists $f\in B_2\setminus B_1$ such that
\[
    B_1\symdiff\{e,f\}\in\cB.
\]
The members of $\cB$ are the \emph{bases} of $M$. The family
\[
    \cI(M)=\{I\subseteq E : I\subseteq B
    \text{ for some } B\in\cB\}
\]
is the family of \emph{independent sets} of $M$.
A set that is not independent is called \emph{dependent}.
A \emph{circuit} of $M$ is a minimal dependent set, and the family
of circuits of $M$ is denoted by $\cC(M)$.
If $B$ is a basis and $e\in E\setminus B$, then $B\cup\{e\}$ contains a unique circuit, denoted $C(e,B)$ and called the \emph{fundamental circuit} of $e$ with respect to $B$.

A matroid $M_R$ on $E$ is \emph{representable over a field $\F$} if
there exists a matrix $R$ over $\F$, whose columns are indexed by $E$,
such that a subset $X\subseteq E$ is independent in $M_R$ if and only if
the columns of $R$ indexed by $X$ are linearly independent over $\F$.
A matroid is \emph{binary} if it is representable over $\Ftwo$.

For a finite set $E$ and a subset $X\subseteq E$, let $\chi_X\in\Ftwo^E$
denote the incidence vector of $X$, defined by
\[
(\chi_X)_e=
\begin{cases}
1, & e\in X,\\
0, & e\notin X.
\end{cases}
\]
We regard incidence vectors as row vectors unless otherwise specified.
Thus, for $X,Y\subseteq E$,
\[
\chi_{X\symdiff Y}=\chi_X+\chi_Y
\]
over $\Ftwo$, and
\[
\chi_X\chi_Y^{\tr}
=
|X\cap Y|\pmod 2.
\]

Let $M$ be a binary matroid on $E$. The \emph{cycle space} $\CS(M)$ of $M$ is the span of the circuits of $M$ under symmetric difference. Each element of $\CS(M)$ is called a cycle. By the standard characterization of binary matroids \cite[Theorem~9.1.2]{Oxley2011}, every cycle can also be expressed
as a disjoint union of circuits. By convention, $\varnothing$ is
the union of zero circuits.
\begin{proposition}
    \label{even}
   Let $M$ be a binary matroid and let $B$ be a basis of $M$.
Then every cycle of $M$ has even cardinality if and only if every
fundamental circuit with respect to $B$ has even cardinality.
\end{proposition}
\begin{proof}
Let $X\in\CS(M)$. Since
\[
C(e,B)\cap(E\setminus B)=\{e\}
\qquad
\text{for every }e\in E\setminus B,
\]
the cycle
\[
X'
=
X\symdiff
\bigtriangleup_{e\in X\setminus B}C(e,B)
\]
is contained in $B$. As $B$ is independent and every nonempty
cycle contains a circuit, we have $X'=\varnothing$. Hence
\[
X=
\bigtriangleup_{e\in X\setminus B}C(e,B).
\]

Therefore, if every fundamental circuit with respect to $B$ has
even cardinality, then every cycle has even cardinality, since
cardinality modulo $2$ is additive under symmetric difference.
The converse is immediate, since every fundamental circuit is a
cycle.
\end{proof}

A matroid is called \emph{bipartite} if every circuit has even
cardinality, in analogy with the characterization of bipartite
graphs by the parity of their cycles.

\subsection{Delta-matroids and vertex flips}\label{subsec:delta}

\begin{definition}[\cite{Bouchet1987}]
A \emph{delta-matroid} is a proper set system $D=(E,\cF)$ satisfying the symmetric exchange axiom: for all $X,Y\in\cF$ and every $u\in X\symdiff Y$, there exists $v\in X\symdiff Y$, possibly $v=u$, such that
\[
X\symdiff\{u,v\}\in\cF.
\]
\end{definition}
 
Let $\cF_{\min}(D)$ and $\cF_{\max}(D)$ be the families of feasible sets of minimum and maximum cardinality. Then
\[
D_{\min}:=(E,\cF_{\min}(D)),
\qquad
D_{\max}:=(E,\cF_{\max}(D))
\]
are matroids, called the \emph{lower} and \emph{upper matroids}, respectively \cite{Bouchet1987}. $D$ is \emph{bipartite} if $D_{\min}$ is a bipartite matroid \cite{YanJin2022}. A delta-matroid $D$ is \emph{even} if all its feasible sets have the same parity
and is \emph{normal} if $\varnothing\in\cF$.

For $X\subseteq E$, the \emph{twist} of $D$ by $X$ is
\[
D*X=\bigl(E,\{F\symdiff X:F\in\cF\}\bigr).
\]
For $e\in E$, the \emph{loop complementation} of $D$ at $e$ is $D+e=(E,\cF')$, where
\[
\cF'=\cF\symdiff\{F\cup\{e\}:F\in\cF,\ e\notin F\}.
\]
Loop complementations on distinct elements commute; hence the
result of applying $+e$ for all $e\in X$ is independent of the
order, and we denote it by $D+X$. The elementary operations
$*e$ and $+e$ are called \emph{vertex flips}; more generally, we
use this term for compositions of such operations
\cite{BrijderHoogeboom2013,BrijderHoogeboom2014}.

We write
\[
D^{*}:=D*E,
\qquad
D^{+}:=D+E,
\]
and use left-associative notation for successive vertex flips:
\[
D^{+*}:=(D+E)*E,
\qquad
D^{*+}:=(D*E)+E.
\]
For a fixed $e$, the operations $*e$ and $+e$ are involutions and generate a group isomorphic to $S_3$ \cite{BrijderHoogeboom2013,BrijderHoogeboom2014}. A delta-matroid is \emph{vf-safe} if every set system obtained
from it by any sequence of vertex flips is again a delta-matroid.

Let $A$ be a symmetric matrix over $\Ftwo$, with rows and columns indexed by $E$. For $X\subseteq E$, let $A[X]$ be its principal submatrix, and regard the empty matrix as nonsingular. Define
\[
D(A):=\bigl(E,\{X\subseteq E:A[X]\text{ is nonsingular}\}\bigr).
\]
The set system $D(A)$ is a delta-matroid \cite{BouchetDuchamp1991}. A delta-matroid $D$ is \emph{binary} if some twist of $D$ is isomorphic to $D(A)$ for a symmetric matrix $A$ over $\Ftwo$. Binary delta-matroids are vf-safe \cite{BrijderHoogeboom2013}.

\begin{lemma}[\cite{BouchetDuchamp1991}]\label{lem:normal-binary-rep}
Let $D$ be a binary delta-matroid and let $F$ be a feasible set. Then $D*F$ is normal, and there is a symmetric matrix $A$ over $\Ftwo$ such that
\[
D*F=D(A).
\]
\end{lemma}
\subsection{Multimatroids, shelter and cycles}\label{subsec:multi}
We recall only the terminology needed here; see \cite{Bouchet1997,Bouchet2001,BrijderHoogeboom2014,Brijder2018} for the general theory.

A \emph{carrier} is a pair $(U,\Omega)$, where $\Omega$ is a partition of $U$. The members of $\Omega$ are \emph{skew classes}; a $2$-element subset of a skew class is a \emph{skew pair}. A set $S\subseteq U$ is a \emph{subtransversal} if $|S\cap\omega|\leq1$ for every $\omega\in\Omega$, and a \emph{transversal} if equality holds for every skew class. The corresponding families are denoted by $\cS(\Omega)$ and $\cT(\Omega)$.

\begin{definition}[\cite{Bouchet1997}]
\normalfont
A \emph{multimatroid} $Z$ (described by its independent sets) is a
triple $(U,\Omega,\mathcal I(Z))$, where $(U,\Omega)$ is a carrier
and $\mathcal I(Z)\subseteq\mathcal S(\Omega)$ satisfies:
\begin{itemize}
\item[(1)] for each transversal $T\in\mathcal T(\Omega)$,
$(T,\mathcal I(Z)\cap 2^T)$ is a matroid described by its
independent sets;
\item[(2)] for every $I\in\mathcal I(Z)$ and every skew pair
$p=\{x,y\}$ contained in a skew class $\omega$ disjoint from $I$,
at least one of $I\cup\{x\}$ and $I\cup\{y\}$ belongs to
$\mathcal I(Z)$.
\end{itemize}
\end{definition}

For $X\subseteq U$, the \emph{restriction} of $Z$ to $X$, denoted
by $Z[X]$, is the multimatroid
\[
\bigl(X,\Omega_X,\mathcal I(Z)\cap2^X\bigr),
\qquad
\Omega_X:=\{\omega\cap X:\omega\in\Omega,\ \omega\cap X\neq\varnothing\}.
\]
We write
\[
Z-X:=Z[U\setminus X]
\]
for deletion of $X$.

Each \( I \in \mathcal{I} \) of \( Z \) is referred to as an \emph{independent set} of \( Z \). 
The set of maximal independent sets \( \max(\mathcal{I}) \) of \( Z \) (with respect to inclusion) 
form the set of \emph{bases}, denoted by \( \mathcal{B}(Z) \). 
A subtransversal is \emph{dependent} if it is not independent.
A \emph{circuit} of $Z$ is an minimal dependent
subtransversal, and the family of circuits is denoted by
$\cC(Z)$. If $C\subseteq T$ for a transversal $T$, then
\(C\in\cC(Z)\) if and only if \(C\in\cC(Z[T]).\)

If every skew class of a multimatroid \( Z \) has cardinality $q$, then $Z$ is a \emph{$q$-matroid}.

\begin{definition}[\cite{Bouchet1997}]
    For $q\geq 2$, a $q$-matroid $Z$  is called \emph{tight} if for every basis \( X \in \mathcal{B}(Z) \) and every skew class \( \omega \in \Omega \), exactly one of the transversals \( (X \setminus \omega) \cup \{u\} \) for \( u \in \omega \) is not a basis of \( Z \).
\end{definition}

Let $M$ be a matroid on $U$. We say that $M$ \emph{shelters} $Z$ if for every \(T\in\cT(\Omega)\),
\(Z[T]=M|T.\) Equivalently, $M$ and $Z$ have the same independent subtransversals. The multimatroid $Z$ is \emph{binary} if it is sheltered by a binary matroid \cite{BrijderTraldi2016,Brijder2018}.

Following Brijder \cite[Section~6.1]{Brijder2018}, define the
family of cycles of a binary multimatroid $Z$ on $(U,\Omega)$ by
\[
\CS(Z):=\bigcup_{T\in\cT(\Omega)}\CS(Z[T]),
\]
where $\CS(Z[T])$ denotes the family of cycles of the matroid
$Z[T]$.

The following observation follows directly from the definition of
sheltering and the cycle space description above.

\begin{proposition}
\label{prop:shelter-cycles}
Let $M$ be a binary matroid that shelters a multimatroid
$Z$ on $(U,\Omega)$. Then
\begin{align}
\cC(Z)
&=
\cC(M)\cap\cS(\Omega),
\label{eq:circuit-shelter}\\
\CS(Z)
&=
\CS(M)\cap\cS(\Omega).
\label{eq:cycle-shelter}
\end{align}
In particular, every cycle of $Z$ is a disjoint union of circuits
of $Z$.
\end{proposition}

\begin{proof}
For every transversal $T$, sheltering gives $Z[T]=M|T$, and hence
\[
\cC(Z[T])=\cC(M|T),
\qquad
\CS(Z[T])=\CS(M|T).
\]
Every subtransversal is contained in a transversal. It follows immediately that the circuits of $Z$ are precisely the subtransversal circuits of $M$, proving \eqref{eq:circuit-shelter}.

If $X\in\CS(Z)$, then $X\in\CS(M|T)$ for some transversal $T$, so $X\in\CS(M)\cap\cS(\Omega)$. Conversely, let $X\in\CS(M)\cap\cS(\Omega)$ and choose a transversal $T$ containing $X$. Write $X$ as a disjoint union of circuits of $M$. Every one of these circuits is contained in $T$ and is therefore a circuit of $M|T=Z[T]$. Thus $X\in\CS(Z[T])\subseteq\CS(Z)$, proving \eqref{eq:cycle-shelter} and the final assertion.
\end{proof}

\subsection{Tight $3$-matroid $Z_3(D,\tau)$}\label{subsec:z3}

Let $E$ be a finite set, and let $Z$ be a tight $3$-matroid with
skew classes
\[
\omega_e=\{e_1,e_2,e_3\},
\qquad e\in E.
\]
For $i\in\{1,2,3\}$, put
\[
T_i:=\{e_i:e\in E\},
\qquad
\tau:=(T_1,T_2,T_3).
\]
For $X\subseteq E$, write
\[
T_i(X):=\{e_i:e\in X\}.
\]
Define a set system $D_{Z,\tau}$ on $E$ by declaring
$X\subseteq E$ feasible if and only if
\[
T_1(E\setminus X)\cup T_2(X)
\]
is a basis of $Z-T_3$. Equivalently, this transversal is
\[
\{e_1:e\in E\setminus X\}
\cup
\{e_2:e\in X\}.
\]
Note that $D_{Z,\tau}$ is a vf-safe delta-matroid \cite{BrijderHoogeboom2014}.

\begin{lemma}[\cite{Bouchet1997}]
\label{lem:tight-even}
Let $Z$ be a tight $3$-matroid. Then
\(Z-T_3\) is tight if and only if \(D_{Z,\tau}\) is even.
\end{lemma}

Conversely, let $D=(E,\cF)$ be a vf-safe delta-matroid. For each
$e\in E$, let
\[
\omega_e=\{e_1,e_2,e_3\},
\]
and put
\[
U:=\bigcup_{e\in E}\omega_e,
\qquad
\Omega:=\{\omega_e:e\in E\}.
\]
For $i\in\{1,2,3\}$, let
\[
T_i:=\{e_i:e\in E\},
\qquad
\tau:=(T_1,T_2,T_3).
\]
Brijder and Hoogeboom \cite{BrijderHoogeboom2014} showed that
there exists a unique tight $3$-matroid $Z$ on $(U,\Omega)$ such
that $D=D_{Z,\tau}.$
We denote this tight $3$-matroid by $Z_3(D,\tau).$

For $e\in E$, with $e_i\in T_i$, define
\[
\tau*e=
\bigl((T_1\setminus\{e_1\})\cup\{e_2\},
      (T_2\setminus\{e_2\})\cup\{e_1\},
      T_3\bigr)
\]
and
\[
\tau+e=
\bigl(T_1,
      (T_2\setminus\{e_2\})\cup\{e_3\},
      (T_3\setminus\{e_3\})\cup\{e_2\}\bigr).
\]
Thus $*e$ interchanges the first two states in $\omega_e$, while $+e$ interchanges the second and third.

\begin{lemma}[\cite{BrijderHoogeboom2014}]
\label{lem:BH-correspondence}
Let $D=(E,\cF)$ be vf-safe delta-matroid and let $Z:=Z_3(D,\tau).$ Then for every $e\in E$,
\[
D*e=D_{Z,\tau*e},
\qquad
D+e=D_{Z,\tau+e}.
\]
\end{lemma}

The operations on distinct skew classes commute, so $\tau*X$ and $\tau+X$ are defined for $X\subseteq E$. On the ground set $E$,
\begin{equation}\label{eq:global-triple-actions}
(T_1,T_2,T_3)\xrightarrow{*E}(T_2,T_1,T_3),
\qquad
(T_1,T_2,T_3)\xrightarrow{+E}(T_1,T_3,T_2).
\end{equation}
Consequently,
\begin{equation}\label{eq:composed-triple-actions}
(T_1,T_2,T_3)\xrightarrow{+E\,*E}(T_3,T_1,T_2),
\qquad
(T_1,T_2,T_3)\xrightarrow{*E\,+E}(T_2,T_3,T_1).
\end{equation}

The following lemma will be used in the proof of the characterization of bipartiteness. For any $X\subseteq E$, let $I_X$ denote the identity matrix indexed by $X$ and $\ind_X$ denote the all-ones row vector indexed by $X$.

\begin{lemma}[{\cite{BrijderTraldi2016II}}]
\label{lem:compatible-shelter}
Let $A$ be a symmetric matrix over $\Ftwo$, with rows and columns
indexed by $E$, and let
\[
D=D(A),\qquad Z=Z_3(D,\tau),
\]
where $\tau=(T_1,T_2,T_3)$. Then $Z$ is sheltered by the binary
matroid $M_R$ represented by
\[
R=
\bordermatrix{
    & T_1 & T_2 & T_3 \cr
    & I_E   & A   & A+I_E
}.
\]

\end{lemma}

\subsection{Orienting transversals}\label{subsec:orienting}

For a tight $3$-matroid $Z$, a transversal $T$ is called \emph{orienting} if $Z-T$ is a tight $2$-matroid. The set of orienting transversals is denoted $\Ort(Z)$. The following theorem of Brijder is the external parity result used in our evenness characterization.

\begin{lemma}[{\cite{Brijder2018}}]\label{lem:brijder-orienting}
If $Z$ is a binary tight
$3$-matroid and $T$ is a transversal, then
\[
T\in\Ort(Z)
\quad\Longleftrightarrow\quad
|Q\cap T|\text{ is even for every }Q\in\CS(Z).
\]
Equivalently, by \cref{prop:shelter-cycles},
$|C\cap T|$ is even for every circuit $C$ of $Z$.
\end{lemma}

\section{Circuit-parity characterizations}\label{sec:parity}

Let $D$ be a binary delta-matroid, $Z=Z_3(D,\tau)$ and $\tau=(T_1,T_2,T_3)$. For each circuit $C$ of $Z$, define
\begin{equation*}\label{eq:pi-def}
p_i(C)=|C\cap T_i|\pmod 2,
\qquad i\in\{1,2,3\}.
\end{equation*}

\subsection{Evenness}\label{subsec:evenness}

\begin{theorem}\label{thm:evenness-parity}
The delta-matroid $D$ is even if and only if
\[
p_3(C)=0,
\qquad\text{for every }C\in\cC(Z).
\]
\end{theorem}

\begin{proof}
By \cref{lem:tight-even}, $D=D_{Z,\tau}$ is even if and only if $Z-T_3$ is tight. By definition, the latter condition says that $T_3$ is orienting. The result now follows from \cref{lem:brijder-orienting}.
\end{proof}

Permuting the transversal triple gives the other two evenness conditions needed later.
\begin{corollary}\label{cor:three-evenness}
Let $Z=Z_3(D,\tau)$ with fixed transversal triple
$\tau=(T_1,T_2,T_3)$. Then
\begin{align*}
D\text{ is even}
&\quad\Longleftrightarrow\quad
p_3(C)=0,
\qquad\text{for all }C\in\cC(Z),\\
D^{+}\text{ is even}
&\quad\Longleftrightarrow\quad
p_2(C)=0,
\qquad\text{for all }C\in\cC(Z),\\
D^{*+}\text{ is even}
&\quad\Longleftrightarrow\quad
p_1(C)=0,
\qquad\text{for all }C\in\cC(Z).
\end{align*}
\end{corollary}

\begin{proof}
The first equivalence is \cref{thm:evenness-parity}.
By \cref{lem:BH-correspondence} and
\eqref{eq:global-triple-actions}--\eqref{eq:composed-triple-actions},
we have
\[
\tau+E=(T_1,T_3,T_2),
\qquad
\tau*E+E=(T_2,T_3,T_1).
\]
Thus the third transversals associated with $D^{+}$ and $D^{*+}$
are, respectively, the original $T_2$ and $T_1$.
The remaining two equivalences now follow from
\cref{thm:evenness-parity}.
\end{proof}

\subsection{The lower matroid in a symmetric representation}\label{subsec:lower-rep}

Before treating bipartiteness, we derive an explicit matrix
representation of the lower matroid. The result is useful independently of the Fano framework.

\begin{proposition}\label{prop:lower-representation}
Let $D$ be binary delta-matroid and let $B$ be a minimum cardinality feasible set. Put $N=E\setminus B$, and choose a symmetric matrix $A$ over $\Ftwo$ such that $D*B=D(A)$. Then
\begin{equation}\label{eq:A-block}
A=\begin{pmatrix}0&P\\P^{\tr}&Q\end{pmatrix}
\end{equation}
with respect to $E=B\cup N$, and
\begin{equation}\label{eq:lower-standard}
D_{\min}=M[I_B\mid P].
\end{equation}
\end{proposition}

\begin{proof}
We first prove that $A[B,B]=0$. If $A_{bb}=1$ for some $b\in B$, then $A[\{b\}]$ is nonsingular, so $\{b\}$ is feasible in $D(A)=D*B$. Hence $B\setminus\{b\}$ is feasible in $D$, contradicting the minimality of $B$.

Now let $b,b'\in B$ be distinct. The diagonal entries at $b$ and $b'$ are zero. If $A_{bb'}=1$, then
\[
A[\{b,b'\}]=\begin{pmatrix}0&1\\1&0\end{pmatrix}
\]
is nonsingular, and $B\setminus\{b,b'\}$ is feasible in $D$, again a contradiction. This proves \eqref{eq:A-block}.

It remains to compare the bases of $D_{\min}$ with those of $M[I_B\mid P]$. Every set $F\subseteq E$ of cardinality $|B|$ can be written uniquely as
\[
F=(B\setminus S)\cup T,
\qquad S\subseteq B,
\quad T\subseteq N,
\quad |S|=|T|.
\]
Since $D*B=D(A)$,
\[
F\in\cF(D)
\quad\Longleftrightarrow\quad
S\cup T\in\cF(D(A))
\quad\Longleftrightarrow\quad
A[S\cup T]\text{ is nonsingular}.
\]
By \eqref{eq:A-block},
\[
A[S\cup T]
=
\begin{pmatrix}
0&P[S,T]\\
P[S,T]^{\tr}&Q[T,T]
\end{pmatrix}.
\]
Because $|S|=|T|$, the matrix above is nonsingular if and only if $P[S,T]$ is nonsingular. Indeed, if \(P[S,T]\) is nonsingular, then every vector in the kernel of \(A[S\cup T]\) must be zero.
Conversely, if $P[S,T]$ is singular, choose
$0\neq x\in\ker P[S,T]^{\tr}$. Then
\[
\binom{x}{0}
\]
is a nonzero vector in the kernel of $A[S\cup T]$.

On the other hand, the columns of $[I_B\mid P]$ indexed by $(B\setminus S)\cup T$ form, after row and column permutations, a matrix of the form
\[
\begin{pmatrix}
I_{B\setminus S}&*\\0&P[S,T]
\end{pmatrix}.
\]
They are linearly independent if and only if $P[S,T]$ is nonsingular. Thus the two matroids have the same bases, proving \eqref{eq:lower-standard}.
\end{proof}

\begin{corollary}\label{cor:bipartite-column-sums}
Under the hypotheses of \cref{prop:lower-representation}, $D$ is bipartite if and only if
\begin{equation}\label{eq:odd-columns}
\sum_{b\in B}A_{be}=1
\qquad\text{for every }e\in N.
\end{equation}
Equivalently, every column of $P$ has odd weight.
\end{corollary}

\begin{proof}
In the representation $[I_B\mid P]$, the fundamental circuit of $e\in N$ with respect to $B$ is
\begin{equation}\label{eq:fundamental-lower}
C_e=\{e\}\cup\{b\in B:A_{be}=1\}.
\end{equation}
Hence
\[
|C_e|\equiv 1+\sum_{b\in B}A_{be}\pmod 2.
\]
By \cref{even}, all circuits of $D_{\min}$ are even if and only if all the $C_e$ are even, which is exactly \eqref{eq:odd-columns}.
\end{proof}

\subsection{Bipartiteness}\label{subsec:bipartite}

We now prove the second parity characterization. 

\begin{theorem}[Bipartiteness parity criterion]
\label{thm:bipartite-parity}
Let $D$ be a binary delta-matroid, let $Z=Z_3(D,\tau),$
and $\tau=(T_1,T_2,T_3)$.
Then $D$ is bipartite if and only if
\[
|C\cap(T_2\cup T_3)|\equiv0\pmod2
\qquad
\text{for every }C\in\cC(Z).
\]
Equivalently,
\[
p_2(C)+p_3(C)=0
\qquad
\text{for every }C\in\cC(Z).
\]
\end{theorem}

\begin{proof}
Choose a basis $B$ of $D_{\min}$ and put $N=E\setminus B$. By \cref{lem:normal-binary-rep}, choose a symmetric matrix $A$ such that $D*B=D(A)$. By \cref{prop:lower-representation}, $A[B,B]=0$ and has the block form \eqref{eq:A-block}.

Let
\[
\tau'=\tau*B.
\]
By \cref{lem:BH-correspondence}, $D_{Z,\tau'}=D*B=D(A).$
Applying \cref{lem:compatible-shelter} to $D(A)$ with respect to
$\tau'$, we see that $Z$ is sheltered by the binary matroid
represented by
\[
\bordermatrix{
    & T'_1 & T'_2 & T'_3 \cr
    & I_E & A & A+I_E
}.
\]

The original set $T_2\cup T_3$ has the following description in the new coordinates:
\begin{equation}\label{eq:old-union-new-coordinates}
T_2\cup T_3
=
T'_1(B)\cup T'_2(N)\cup T'_3(E).
\end{equation}
Indeed, the first two states are interchanged on $B$ and unchanged on $N$, while the third state is never changed.

Assume first that $D$ is bipartite. Recall $\chi_S$ denotes the incidence row vector of $S$. Let $x=\chi_B\in\Ftwo^E$. Since $A[B,B]=0$,
\[
xI_E=(\ind_B,0_N),
\]
where $\ind_B$ is the all-ones row vector indexed by $B$. By \cref{cor:bipartite-column-sums},
\[
xA=(0_B,\ind_N).
\]
Consequently,
\[
x(A+I_E)=\ind_E.
\]
Together with \eqref{eq:old-union-new-coordinates}, these identities say that
\begin{equation}\label{eq:rowspace-indicator}
xR=\chi_{T_2\cup T_3}.
\end{equation}

Let $C$ be a circuit of $Z$. By \cref{prop:shelter-cycles}, $C$ is a circuit of the binary matroid $M_R$. A minimally dependent set of columns over $\Ftwo$ has the unique nonzero dependence with every coefficient equal to $1$, and hence
\begin{equation}\label{eq:circuit-column-sum}
\sum_{u\in C}R_u=0,
\end{equation}
where $R_u$ is the column of $R$ indexed by $u$.
Using \eqref{eq:rowspace-indicator} and \eqref{eq:circuit-column-sum},
\[
|C\cap(T_2\cup T_3)|
\equiv
\chi_{T_2\cup T_3}\chi_C^{\tr}
=
xR\chi_C^{\tr}
=
x\left(\sum_{u\in C}R_u\right)
=0.
\]
This proves the forward implication.

Conversely, suppose that every circuit of $Z$ meets $T_2\cup T_3$ evenly. Fix $e\in N$. Among the two columns $A_e$ and $(A+I_E)_e$, choose
\[
r_e=
\begin{cases}
A_e,&A_{ee}=0,\\
(A+I_E)_e,&A_{ee}=1.
\end{cases}
\]
Thus the $e$th coordinate of $r_e$ is zero. Relative to the basis $T'_1$, represented by the identity block, the fundamental circuit of $r_e$ in the $M_R$ is
\begin{equation}\label{eq:Ke}
K_e=\{r_e\}\cup\{T'_1(f):(r_e)_f=1\}.
\end{equation}
The equality $(r_e)_e=0$ ensures that $K_e$ does not contain $T'_1(e)$ in addition to $r_e$. Hence $K_e$ is a subtransversal, and \cref{prop:shelter-cycles} implies that it is a circuit of $Z$.

Since $e\in N$, the chosen element $r_e$ lies in
$T_2\cup T_3$ and hence contributes one element to
$K_e\cap(T_2\cup T_3)$. The remaining elements of this
intersection come from the identity block and are precisely
those $T_1'(b)$ with $b\in B$ for which $(r_e)_b=1$.
Therefore
\[
|K_e\cap(T_2\cup T_3)|
\equiv
1+\sum_{b\in B}(r_e)_b
\pmod 2.
\]
Since $b\neq e$ for every $b\in B$, the $b$th coordinate of
$r_e$ is $A_{be}$ whether $r_e=A_e$ or
$r_e=(A+I_E)_e$. Hence
\[
|K_e\cap(T_2\cup T_3)|
\equiv
1+\sum_{b\in B}A_{be}
\pmod 2.
\]

The left-hand side is zero by hypothesis, so \eqref{eq:odd-columns} holds for every $e\in N$. By \cref{cor:bipartite-column-sums}, $D$ is bipartite.
\end{proof}

As with evenness, the remaining bipartiteness conditions are obtained by permuting the transversal triple.

\begin{corollary}\label{cor:three-bipartite}
For every circuit $C$ of $Z$,
\begin{align*}
D\text{ is bipartite}
&\quad\Longleftrightarrow\quad p_2(C)+p_3(C)=0\text{ for all }C,\\
D^{*}\text{ is bipartite}
&\quad\Longleftrightarrow\quad p_1(C)+p_3(C)=0\text{ for all }C,\\
D^{+*}\text{ is bipartite}
&\quad\Longleftrightarrow\quad p_1(C)+p_2(C)=0\text{ for all }C.
\end{align*}
\end{corollary}

\begin{proof}
The first assertion is \cref{thm:bipartite-parity}.
By \cref{lem:BH-correspondence} and
\eqref{eq:global-triple-actions}--\eqref{eq:composed-triple-actions},
we have
\[
\tau*E=(T_2,T_1,T_3),
\qquad
(\tau+E)*E=(T_3,T_1,T_2).
\]
The remaining two equivalences now follow from
\cref{thm:bipartite-parity}.
\end{proof}

\section{The Fano framework}\label{sec:fano}

\begin{definition}\label{def:z3-bipartite}
A binary delta-matroid $D$ is \emph{$Z_3$-bipartite} if every circuit of $Z_3(D,\tau)$ has even cardinality.
\end{definition}

Because every circuit of $Z$ is a subtransversal and $U(Z)=T_1\cup T_2\cup T_3$, $Z_3$-bipartiteness is equivalent to
\[
p_1(C)+p_2(C)+p_3(C)=0
\qquad\text{for every circuit }C.
\]
Thus the seven properties are summarized by \cref{tab:dictionary}.

\begin{table}[ht]
\centering
\caption{The seven properties as circuit-parity conditions. Every displayed equation is required to hold for all circuits $C$ of $Z_3(D,\tau)$.}\label{tab:dictionary}
\begin{tabularx}{\textwidth}{@{}cYY@{}}
\toprule
Vector & Delta-matroid property & Circuit-parity condition \\
\midrule
$001$ & $D$ is even & $p_3(C)=0$ \\
$010$ & $D^{+}$ is even & $p_2(C)=0$ \\
$011$ & $D$ is bipartite & $p_2(C)+p_3(C)=0$ \\
$100$ & $D^{*+}$ is even & $p_1(C)=0$ \\
$101$ & $D^{*}$ is bipartite & $p_1(C)+p_3(C)=0$ \\
$110$ & $D^{+*}$ is bipartite & $p_1(C)+p_2(C)=0$ \\
$111$ & $D$ is $Z_3$-bipartite & $p_1(C)+p_2(C)+p_3(C)=0$ \\
\bottomrule
\end{tabularx}
\end{table}

For $\beta=(\beta_1,\beta_2,\beta_3)\in\Ftwo^3$, let $P_\beta(D)$ denote the condition
\begin{equation}\label{eq:Pbeta}
\beta_1p_1(C)+\beta_2p_2(C)+\beta_3p_3(C)=0
\qquad\text{for every }C\in\cC(Z_3(D,\tau)).
\end{equation}
For $\beta\neq 000$, this is exactly the corresponding property in \cref{tab:dictionary}; $P_{000}$ is the always-true property.

\begin{theorem}[Fano framework]\label{thm:fano-framework}
Let $D$ be a binary delta-matroid and define
\[
X(D)=\{\beta\in\Ftwo^3:P_\beta(D)\text{ holds}\}.
\]
Then $X(D)$ is a vector subspace of $\Ftwo^3$. Consequently, the satisfied Fano points are $X(D)\setminus\{000\}$, the projectivization of $X(D)$, which may be empty; see \cref{fig:fano}.
\end{theorem}

\begin{proof}
Clearly $0\in X(D)$. Let $\beta,\gamma\in X(D)$. Then, for every
circuit $C$ of $Z_3(D,\tau)$,
\[
\beta_1p_1(C)+\beta_2p_2(C)+\beta_3p_3(C)=0
\]
and
\[
\gamma_1p_1(C)+\gamma_2p_2(C)+\gamma_3p_3(C)=0.
\]
Adding these equalities in $\Ftwo$ gives
\[
(\beta_1+\gamma_1)p_1(C)
+(\beta_2+\gamma_2)p_2(C)
+(\beta_3+\gamma_3)p_3(C)=0.
\]
Hence $\beta+\gamma\in X(D)$. Therefore $X(D)$ contains the zero
vector and is closed under addition, so $X(D)$ is a vector subspace
of $\Ftwo^3$.
\end{proof}

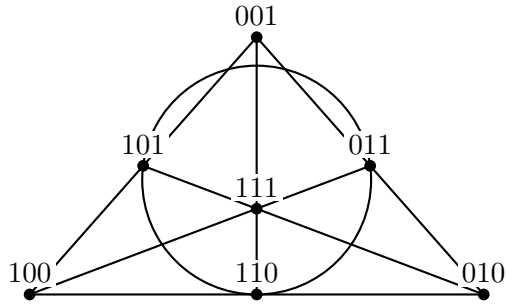
\begin{figure}[ht]
\centering
\begin{tikzpicture}[scale=1.0, every node/.style={font=\small}]
  \coordinate (A) at (0,3.4);
  \coordinate (B) at (-3,0);
  \coordinate (C) at (3,0);
  \coordinate (AB) at (-1.5,1.7);
  \coordinate (AC) at (1.5,1.7);
  \coordinate (BC) at (0,0);
  \coordinate (O) at (0,1.133333); 
  \draw[thick] (A)--(B)--(C)--cycle;
  \draw[thick] (A)--(BC);
  \draw[thick] (B)--(AC);
  \draw[thick] (C)--(AB);
  \draw[thick] (0,1.512) circle[radius=1.512];
  \foreach \P/\lab in {A/001,B/100,C/010,AB/101,AC/011,BC/110,O/111}
    {\fill (\P) circle (2.2pt); \node[fill=white,inner sep=1.5pt] at ($ (\P)+(0,0.28) $) {$\lab$};}
\end{tikzpicture}
\caption{The Fano plane labelled by the seven delta-matroid properties.}\label{fig:fano}
\end{figure}

The standard consequences of \cref{thm:fano-framework} take the following form.

\begin{corollary}[The seven Fano lines]\label{cor:seven-lines}
In each of the following triples, any two properties imply the third:
\begin{align*}
&\{001,010,011\}, &&\{001,100,101\}, &&\{001,110,111\},\\
&\{010,100,110\}, &&\{010,101,111\}, &&\{011,100,111\},\\
&{\{011,101,110\}.}
\end{align*}
Equivalently, these are
\begin{align*}
&\{D\text{ even},\ D^{+}\text{ even},\ D\text{ bipartite}\},\\
&\{D\text{ even},\ D^{*+}\text{ even},\ D^{*}\text{ bipartite}\},\\
&\{D\text{ even},\ D^{+*}\text{ bipartite},\ D\text{ $Z_3$-bipartite}\},\\
&\{D^{+}\text{ even},\ D^{*+}\text{ even},\ D^{+*}\text{ bipartite}\},\\
&\{D^{+}\text{ even},\ D^{*}\text{ bipartite},\ D\text{ $Z_3$-bipartite}\},\\
&\{D\text{ bipartite},\ D^{*+}\text{ even},\ D\text{ $Z_3$-bipartite}\},\\
&\{D\text{ bipartite},\ D^{*}\text{ bipartite},\ D^{+*}\text{ bipartite}\}.
\end{align*}
\end{corollary}

\begin{proof}
Three nonzero vectors form a line of the Fano plane precisely when their sum is zero. If two belong to the subspace $X(D)$, so does their sum, which is the third.
\end{proof}

\begin{corollary}\label{cor:number-properties}
A binary delta-matroid satisfies exactly $0$, $1$, $3$ or $7$ of the seven nontrivial properties.
\end{corollary}

\begin{proof}
A subspace of $\Ftwo^3$ has dimension $0$, $1$, $2$ or $3$, and hence has respectively $0$, $1$, $3$ or $7$ nonzero vectors.
\end{proof}

\begin{corollary}\label{cor:noncollinear}
If a binary delta-matroid satisfies three properties whose vectors are not collinear in the Fano plane, then it satisfies all seven properties.
\end{corollary}

\begin{proof}
Three noncollinear points span $\Ftwo^3$, and $X(D)$ is a subspace containing them.
\end{proof}

\begin{remark}\label{rem:yan-jin}
Yan and Jin proved that, for a binary even delta-matroid, $D$ is bipartite if and only if $D^{+}$ is even \cite{YanJin2022}. This is the line $\{001,010,011\}$ under the assumption that $001$ holds.
\end{remark}

\section{The ribbon-graphic specialization}
\label{sec:ribbon-specialization}

We now show that the Fano framework for binary delta-matroids
specializes to the graph-embedding framework of Dunshee and
Ellingham \cite{DunsheeEllingham2025}.

Throughout this section, all ribbon graphs are assumed to be
connected and to have at least one edge. The disconnected case may
be recovered componentwise.

\subsection{Ribbon graphs}
\label{subsec:ribbon}

A cellularly embedded graph determines a ribbon graph by taking a
closed regular neighbourhood of the embedded graph. Conversely,
capping each boundary component of a ribbon graph with a disc gives
a cellular embedding in a closed surface.

\begin{definition}[\cite{Bollobas2002}]
		\normalfont
A \emph{ribbon graph} \( G = (V(G), E(G)) \) is a surface with boundary, represented as the union of two sets of topological discs: a set \( V(G) \) of vertices and a set \( E(G) \) of edges, satisfying the following properties:
\begin{enumerate}[label=\textup{(\alph*)}]
\item  The vertices and edges intersect in disjoint line segments.
\item   Each such line segment lies on the boundary of exactly one vertex and exactly one edge.
\item   Every edge contains exactly two such line segments.
\end{enumerate}
\end{definition}

A \emph{spanning ribbon subgraph} of $G$ is obtained by deleting
edges while keeping all vertex-discs. A connected ribbon graph is a \emph{quasi-tree} if it has exactly
one boundary component. A \emph{spanning quasi-tree} of a connected
ribbon graph $G$ is a spanning ribbon subgraph of $G$ that is a
quasi-tree. A ribbon graph is \emph{orientable} if its underlying surface with
boundary is orientable.

We next recall the two ribbon-graph operations corresponding to
twist and loop complementation of delta-matroids.

\begin{definition}[{\cite{Chmutov2009}}]
\normalfont
Let $G$ be a ribbon graph and let $A\subseteq E(G)$. The
\emph{partial dual} of $G$ with respect to $A$, denoted by
$G^{\delta|A}$, is obtained as follows. Attach a disc to every
boundary component of the spanning ribbon subgraph $(V(G),A)$.
These discs become the vertex-discs of $G^{\delta|A}$. Then remove
the interiors of the original vertex-discs, leaving the edge-ribbons
unchanged.
\end{definition}

The \emph{geometric dual} of $G$ is the full partial dual
\[
G^*:=G^{\delta|E(G)}.
\]

\begin{definition}[{\cite{EllisMonaghanMoffatt2012}}]
\normalfont
Let $G$ be a ribbon graph and let $A\subseteq E(G)$. The
\emph{partial Petrial} of $G$ with respect to $A$, denoted by
$G^{\tau|A}$, is obtained by adding a half-twist to every
edge-ribbon in $A$.
\end{definition}

The \emph{Petrial} of $G$ is
\[
G^\times:=G^{\tau|E(G)}.
\]
Recall that, in this paper, compositions are read from left to right; for example,
\[
G^{\delta\tau|A}
=
\bigl(G^{\delta|A}\bigr)^{\tau|A},
\qquad
G^{\tau\delta|A}
=
\bigl(G^{\tau|A}\bigr)^{\delta|A}.
\]

For comparison with the Fano properties, recall that an embedded
graph is \emph{directable} if its edges can be oriented so that every
face boundary is a directed closed walk. It is
\emph{$2$-face-colourable} if its faces admit a proper
$2$-colouring.

\subsection{The ribbon-graphic delta-matroid}
\label{subsec:ribbon-delta}

Let $G$ be a connected ribbon graph and define
\[
\cF(G)
:=
\{A\subseteq E(G):
  A\text{ is the edge set of a spanning quasi-tree of }G\}.
\]

\begin{definition}
\normalfont
The set system
\[
D(G):=(E(G),\cF(G))
\]
is called the \emph{ribbon-graphic delta-matroid} of $G$.
\end{definition}

It is a binary delta-matroid
\cite{Bouchet1989,ChunEtAl2019JCTA}. We shall use the following standard correspondence between
ribbon-graph operations and delta-matroid operations.

\begin{lemma}\cite{ChunEtAl2019JCTA,ChunEtAl2019PLMS}.
\label{lem:ribbon-dictionary}
Let $G$ be a connected ribbon graph and let $A\subseteq E(G)$. Then
\begin{align}
D(G^{\delta|A}) &= D(G)*A,
\label{eq:ribbon-dual}\\
D(G^{\tau|A}) &= D(G)+A.
\end{align} Moreover,
\[
D(G)\text{ is even}
\quad\Longleftrightarrow\quad
G\text{ is orientable}.
\]
\end{lemma}

Combining \cref{lem:ribbon-dictionary} with the embedded-graph
equivalences in \cite{DunsheeEllingham2025} yields the first six
identifications in the seven-point Fano framework. We include all
seven points in \cref{tab:ribbon-dictionary}; the final
identification, corresponding to $111$, is established in
\cref{cor:ribbon-111}. Here $\Sigma$ denotes the ambient surface
containing the natural simultaneous embedding of $G$ and its
geometric dual $G^*$.

\begin{table}[ht]
\centering
\caption{The seven Fano points in the ribbon-graphic case.}
\label{tab:ribbon-dictionary}
\begin{tabularx}{\textwidth}{@{}cYY@{}}
\toprule
Point
& Delta-matroid property
& Embedded-graph property\\
\midrule
$001$
& $D(G)$ is even
& $G$ is orientable\\

$010$
& $D(G)^+$ is even
& $G^\times$ is orientable\\

$011$
& $D(G)$ is bipartite
& $G$ is bipartite\\

$100$
& $D(G)^{*+}$ is even
& $G^{*\times}$ is orientable, equivalently $G$ is directable\\

$101$
& $D(G)^*$ is bipartite
& $G^*$ is bipartite, equivalently $G$ is $2$-face-colourable\\

$110$
& $D(G)^{+*}$ is bipartite
& $G^{\times *}$ is bipartite, equivalently the regions of
  $\Sigma\setminus(G\cup G^*)$ are $2$-colourable.\\

{$111$}
& {$D(G)$ is $Z_3$-bipartite}
& {the medial graph $G_m$ of $G$ is bipartite.}\\
\bottomrule
\end{tabularx}
\end{table}

\subsection{The medial graph and transitions}
\label{subsec:transition}

A \emph{corner} of a ribbon graph is determined by two consecutive
edge-ends on the boundary of a vertex-disc.

\begin{definition}[{\cite{EllisMonaghanMoffatt2012}}]
\normalfont
Let $G$ be a connected ribbon graph. Its \emph{medial graph}
$G_m$ is obtained by placing a vertex $v_e$ on each edge
$e\in E(G)$ and joining these vertices by following the boundary
of $G$ through its corners.
\end{definition}

Each edge of $G$ is incident with four corners. Consequently every
vertex of $G_m$ has degree $4$. Since $G$ is connected, $G_m$ is
also connected. Hence $G_m$ is a connected $4$-regular graph, and
there is a natural bijection
\[
E(G)\longleftrightarrow V(G_m),
\qquad
e\longmapsto v_e.
\]
The medial graph $G_m$ has a canonical checkerboard colouring:
the faces containing the vertex-discs of $G$ are coloured black
and the remaining faces white.

Let $F$ be a $4$-regular graph and let $v\in V(F)$. A
\emph{transition at $v$} is a partition of the four half-edges
incident with $v$ into two pairs. There are exactly three transitions at each vertex. Let
$\mathfrak T(F)$ denote the set of all transitions of $F$, and let $\omega_v$
denote the set of the three transitions at $v$. Then
\[
\Omega(F):=\{\omega_v:v\in V(F)\}
\]
is a partition of $\mathfrak T(F)$ into $3$-element skew classes.

\subsection{The transition $3$-matroid}
\label{subsec:transition-matroid}
We begin with a $4$-regular graph to construct a $3$-matroid.
Let $F$ be a connected $4$-regular graph. Fix an Euler tour
$\Gamma$ of $F$, that is, a closed trail traversing every edge
exactly once. Each vertex occurs twice in the cyclic vertex
sequence of $\Gamma$.

Two distinct vertices $u,v\in V(F)$ are \emph{interlaced with
respect to $\Gamma$} if their occurrences alternate along
$\Gamma$. The \emph{interlacement graph} $I(\Gamma)$ is the simple
graph on $V(F)$ in which two vertices are adjacent precisely when
they are interlaced. Let
\[
A:=A(I(\Gamma))
\]
be its adjacency matrix over $\Ftwo$. Thus
\[
A=A^{\tr}
\qquad\text{and}\qquad
A_{vv}=0
\quad
\text{for every }v\in V(F).
\]

Temporarily orient $\Gamma$. At each vertex $v$, the tour $\Gamma$ passes through $v$
twice. Let $a_1,b_1,a_2,b_2$ be the four half-edges incident with
$v$, labelled so that the two directed passages through $v$ are
\[
a_1\,v\,b_1,
\qquad
a_2\,v\,b_2.
\]
Following Traldi \cite[Definition~12]{Traldi2015}, define
\[
\begin{aligned}
\phi_\Gamma(v)
&=\{\{a_1,b_1\},\{a_2,b_2\}\},\\
\kappa_\Gamma(v)
&=\{\{a_1,b_2\},\{a_2,b_1\}\},\\
\psi_\Gamma(v)
&=\{\{a_1,a_2\},\{b_1,b_2\}\}.
\end{aligned}
\]

Put
\[
T_\phi:=\{\phi_\Gamma(v):v\in V(F)\},\qquad
T_\kappa:=\{\kappa_\Gamma(v):v\in V(F)\},\]
\[
T_\psi:=\{\psi_\Gamma(v):v\in V(F)\},
\]
and let
\[
\tau_\Gamma:=(T_\phi,T_\kappa,T_\psi).
\]

The \emph{transition matroid} of $F$, denoted by $M_\tau(F)$, is
the binary matroid on $\mathfrak T(F)$ represented by
\[
R_\Gamma
=\bordermatrix{
    &T_\phi & T_\kappa & T_\psi \cr
    & I_E   & A   & A+I_E
}.
\]
Although $R_\Gamma$ depends on the chosen Euler tour $\Gamma$,
the matroid it represents on $\mathfrak T(F)$ does not
\cite[Proposition~15]{Traldi2015}. We denote this matroid by
$M_\tau(F)$.

For a transition transversal $T\in\cT(\Omega(F))$, the transitions selected by $T$ join the edges of $F$ into edge-disjoint
closed trails whose edge sets partition $E(F)$. Let $P_T$ denote
this circuit partition. Thus $|P_T|$ is the number of closed trails
in the partition, and $|P_T|=1$ exactly when the selected
transitions join all edges into a single closed trails of $F$.

Since $A$ is symmetric over $\Ftwo$, apply
\cref{lem:compatible-shelter} to $D(A)$ and the transversal triple
$\tau_\Gamma$. The transition matroid
$M_\tau(F)$ shelters $Z_3(D(A),\tau_\Gamma)$, and we set
\[
Z(F):=Z_3(D(A),\tau_\Gamma).
\]
Thus
\[
Z(F)[T]=M_\tau(F)|T
\qquad
\text{for every }T\in\cT(\Omega(F)).
\]
For later use, Traldi's circuit-nullity formula \cite[Theorem~1]{Traldi2015} gives
\begin{equation}
\label{eq:transition-basis}
T\in\mathcal B(Z(F))
\quad\Longleftrightarrow\quad
|P_T|=1
\qquad
\text{for every }T\in\cT(\Omega(F)),
\end{equation}
because $F$ is connected.

Let $G_m$ be the medial graph of a connected ribbon graph $G$.
For $e\in E(G)$, let
\[
b_e:=b(v_e),\qquad
w_e:=w(v_e),\qquad
c_e:=c(v_e)
\]
denote, respectively, the black smoothing, the white smoothing,
and the crossing at the corresponding medial vertex $v_e$. Put
\[
T_b:=\{b_e:e\in E(G)\},\qquad
T_w:=\{w_e:e\in E(G)\},\qquad
T_c:=\{c_e:e\in E(G)\},
\]
and set
\[
\tau_G:=(T_b,T_w,T_c).
\]

Using the natural bijection
\[
E(G)\longleftrightarrow V(G_m),
\qquad
e\longmapsto v_e,
\]
we identify the skew class of $Z_3(D(G),\tau_G)$ indexed by $e$
with
\[
\omega_{v_e}=\{b_e,w_e,c_e\}.
\]
Thus $Z_3(D(G),\tau_G)$ and $Z(G_m)$ are regarded as
$3$-matroids on the same carrier
\[
(\mathfrak T(G_m),\Omega(G_m)).
\]

For an ordered partition
\[
E(G)=X\cup Y\cup W,
\]
define
\[
T(X,Y,W)
:=
\{w_e:e\in X\}
\cup
\{b_e:e\in Y\}
\cup
\{c_e:e\in W\}.
\]
The three choices correspond, respectively, to retaining an edge,
deleting it, and retaining it after adding a half-twist. Hence the
corresponding ribbon-graph state is
\[
G^{\tau|W}\setminus Y.
\]
At every edge, its boundary arcs are paired in the same way as the
medial half-edges prescribed by $T(X,Y,W)$. Therefore the boundary
components of $G^{\tau|W}\setminus Y$ are naturally in bijection
with the members of the circuit partition
$P_{T(X,Y,W)}$.
\begin{lemma}
\label{lem:medial-z3}
Let $G$ be a connected ribbon graph. Under the identification
above,
\[
Z_3(D(G),\tau_G)=Z(G_m).
\]
\end{lemma}

\begin{proof}
For the ordered partition $E(G)=X\cup Y\cup W$, put
$T=T(X,Y,W)$. By
\cite[Theorem~6.3]{MerinoMoffattNoble2025}, the transversal $T$ is
a basis of $Z_3(D(G),\tau_G)$ if and only if the ribbon-graph state
$G^{\tau|W}\setminus Y$ has one boundary component. By above, this is equivalent to
$|P_T|=1$. By \eqref{eq:transition-basis}, the latter condition is
equivalent to $T\in\mathcal B(Z(G_m))$. Hence $Z_3(D(G),\tau_G)$ and $Z(G_m)$ have the same bases.
Therefore, as $3$-matroids on the same carrier,
\[Z_3(D(G),\tau_G)=Z(G_m).\]
\end{proof}

It remains to determine when every circuit of $Z(F)$ has even
cardinality.

\subsection{Bipartiteness of connected $4$-regular graphs}
\label{subsec:4regular}
We first relate bipartiteness of $F$ to the vertex degrees of its
interlacement graph.

\begin{lemma}
\label{lem:odd-interlacement-degrees}
Let $F$ be a connected $4$-regular graph, let $\Gamma$ be an Euler
tour of $F$, and let
\[
H=I(\Gamma).
\]
Then
\[
F\text{ is bipartite}
\quad\Longleftrightarrow\quad
\deg_H(v)\text{ is odd for every }v\in V(F).
\]
\end{lemma}

\begin{proof}
Fix $v\in V(F)$. The two occurrences of $v$ in the cyclic vertex
sequence of $\Gamma$ divide $\Gamma$ into two closed
$v$--$v$ subtrails. Choose one of them, and let $k$ be the number
of vertex occurrences strictly between its two occurrences of
$v$. This subtrail has length $k+1$.

For $w\neq v$, the two occurrences of $w$ contribute an odd number
to $k$ precisely when exactly one of them lies strictly between the
two chosen occurrences of $v$. This happens exactly when $w$ is
interlaced with $v$. Consequently,
\[
k\equiv\deg_H(v)\pmod 2,
\]
and hence the length of the chosen closed subtrail is congruent to
\[
\deg_H(v)+1\pmod 2.
\]

Suppose first that $F$ is bipartite. Every closed trail in a
bipartite graph has even length, so
\[
\deg_H(v)+1\equiv0\pmod2.
\]
Thus every vertex of $H$ has odd degree.

Conversely, suppose that every vertex of $H$ has odd degree. Write
the Euler tour cyclically as
\[
\Gamma
=
v_0e_1v_1e_2\cdots e_{m-1}v_{m-1}e_mv_0,
\]
where $m=|E(F)|$.
Since $F$ is $4$-regular, $m=2|V(F)|$ is even.

Assign colour $i\pmod2$ to the occurrence of $v_i$ in the cyclic
sequence, for $0\leq i<m$. We show that this gives a well-defined
colouring of the vertices of $F$.

Suppose that a vertex $v$ occurs at positions $i$ and $j$, with
$i<j$. The subtrail from the occurrence at position $i$ to the
occurrence at position $j$ has length $j-i$. By the parity
calculation above,
\[
j-i
\equiv
\deg_H(v)+1
\equiv0
\pmod2.
\]
Thus $i$ and $j$ have the same parity, so the two occurrences of
$v$ receive the same colour.

Every edge of $F$ occurs on $\Gamma$ between two consecutive vertex
occurrences, whose positions have opposite parity. Hence the two
ends of every edge receive different colours. Therefore there is a
proper $2$-colouring of $F$, and $F$ is bipartite.
\end{proof}

We now characterize bipartiteness in terms of the transition
$3$-matroid.

\begin{theorem}
\label{thm:transition-bipartite}
Let $F$ be a connected $4$-regular graph. Then
\[
F\text{ is bipartite}
\quad\Longleftrightarrow\quad
\text{every circuit of }Z(F)\text{ has even cardinality}.
\]
\end{theorem}

\begin{proof}
Let $\Gamma$ be an Euler tour of $F$, let
\[
H:=I(\Gamma),
\qquad
A:=A(H),
\]
and retain the matrix $R_\Gamma$ from
\cref{subsec:transition-matroid}. By the definition of $Z(F)$,
the transition matroid $M_\tau(F)$ represented by $R_\Gamma$
shelters $Z(F)$.

Suppose first that every circuit of $Z(F)$ has even cardinality.
Fix $v\in V(F)$. The identity block of $R_\Gamma$ is a basis of
$M_\tau(F)$, and the fundamental circuit of the column indexed by
$\kappa_\Gamma(v)$ with respect to this basis is
\[
K_v
=
\{\kappa_\Gamma(v)\}
\cup
\{\phi_\Gamma(u):A_{uv}=1\}.
\]
Since $A_{vv}=0$, $K_v$ does not contain $\phi_\Gamma(v)$ and therefore
contains at most one transition from each skew class. Thus $K_v$
is a subtransversal.

The set $K_v$ is a circuit of $M_\tau(F)$, so
\cref{prop:shelter-cycles} implies that it is a circuit of $Z(F)$.
Consequently,
\[
0
\equiv
|K_v|
=
1+\deg_H(v)
\pmod2.
\]
Thus every vertex of $H$ has odd degree. By
\cref{lem:odd-interlacement-degrees}, $F$ is bipartite.

Conversely, suppose that $F$ is bipartite. By
\cref{lem:odd-interlacement-degrees}, every vertex of $H$ has odd
degree. Hence every column of $A$ has odd weight.
Let $\mathbf1$ denote the all-ones row vector indexed by $V(F)$.
Then
\begin{equation}
\label{eq:ones-A}
\mathbf1 A=\mathbf1.
\end{equation}
Let $K$ be a circuit of $Z(F)$, and define
\[
X_\phi
:=
\{v\in V(F):\phi_\Gamma(v)\in K\}, \qquad X_\kappa
:=
\{v\in V(F):\kappa_\Gamma(v)\in K\},
\]
\[
X_\psi
:=
\{v\in V(F):\psi_\Gamma(v)\in K\}.
\]
Let
\[
x:=\chi_{X_\phi}^{\tr},
\qquad
y:=\chi_{X_\kappa}^{\tr},
\qquad
z:=\chi_{X_\psi}^{\tr}.
\]
Since $K$ is a subtransversal, the sets
\[
X_\phi,\qquad X_\kappa,\qquad X_\psi
\]
are pairwise disjoint.

By \cref{prop:shelter-cycles}, $K$ is also a circuit of
$M_\tau(F)$. Since $M_\tau(F)$ is binary, the columns of
$R_\Gamma$ indexed by $K$ sum to zero. Hence
\begin{equation}
\label{eq:xyz-dependence}
x+Ay+(A+I_E)z=0.
\end{equation}

Put
\[
w:=y+z.
\]
Then \eqref{eq:xyz-dependence} becomes
\[
x+Aw+z=0.
\]
Multiplying on the left by $w^{\tr}$ gives
\[
w^{\tr}x+w^{\tr}Aw+w^{\tr}z=0.
\]

The supports of $w$ and $x$ are disjoint, so
\[
w^{\tr}x=0.
\]
Since $A$ is symmetric with zero diagonal,
\[
w^{\tr}Aw=0
\]
over $\Ftwo$. Moreover, the supports of $y$ and $z$ are disjoint,
and hence
\[
w^{\tr}z
=
(y+z)^{\tr}z
=
|X_\psi|
\pmod2.
\]
It follows that
\[
|X_\psi|\equiv0\pmod2.
\]

Now multiply \eqref{eq:xyz-dependence} on the left by $\mathbf1$.
Using \eqref{eq:ones-A}, we obtain
\[
\begin{aligned}
0
&=
\mathbf1 x
+\mathbf1 Ay
+\mathbf1(A+I)z\\
&=
|X_\phi|+|X_\kappa|
\pmod2.
\end{aligned}
\]
Together with the evenness of $|X_\psi|$, this gives
\[
|K|
=
|X_\phi|+|X_\kappa|+|X_\psi|
\equiv0\pmod2.
\]
Therefore every circuit of $Z(F)$ has even cardinality.
\end{proof}

\begin{corollary}
\label{cor:ribbon-111}
Let $G$ be a connected ribbon graph. Then
\[
D(G)\text{ is $Z_3$-bipartite}
\quad\Longleftrightarrow\quad
G_m\text{ is bipartite}.
\]
\end{corollary}

\begin{proof}
By \cref{lem:medial-z3},
\[
Z_3(D(G),\tau_G)\cong Z(G_m).
\]
Therefore
\[
\begin{aligned}
D(G)\text{ is $Z_3$-bipartite}
&\iff
\text{every circuit of }Z(G_m)\text{ has even cardinality}\\
&\iff
G_m\text{ is bipartite},
\end{aligned}
\]
where the second equivalence follows from
\cref{thm:transition-bipartite}.
\end{proof}

\section{A vf-safe counterexample}
\label{sec:vfsafe-counterexample}

This resolves the binary case of
Problem~\ref{prob:dunshee-ellingham}: the tight $3$-matroid
$Z_3(D,\tau)$ provides a common parity model for the seven
properties, and their Fano structure follows from linear algebra
over $\Ftwo$. The binary hypothesis cannot in general be replaced
by vf-safe. As the following example shows, even the
Fano-line implication on $\{011,101,110\}$ fails for a vf-safe
delta-matroid.

\begin{example}
\label{ex:vfsafe-counterexample}
Let $M=U_{2,4}$ on $E=\{1,2,3,4\},$
regarded as a delta-matroid whose feasible sets are the
$2$-subsets of $E$. The matroid $U_{2,4}$ is quaternary and hence
vf-safe \cite{BrijderHoogeboomQuaternary}. Define, with vertex
flips composed from left to right,
\[
D:=M*1+1*2+3.
\]
Using the shorthand $14=\{1,4\}$, a direct calculation gives
\[
\mathcal F(D)
=
\{
\varnothing,1,3,14,23,24,123,124,134,234
\}.
\]

The minimum feasible sets of the three relevant delta-matroids are
\[
\mathcal F_{\min}(D)=\{\varnothing\},
\]
and
\[
\mathcal F_{\min}(D^*)
=
\mathcal F_{\min}(D^{+*})
=
\bigl\{
\{1\},\{2\},\{3\},\{4\}
\bigr\}.
\]
Consequently,
\[
D_{\min}=U_{0,4},
\qquad
(D^*)_{\min}=(D^{+*})_{\min}=U_{1,4}.
\]

The circuits of $U_{0,4}$ are the singletons, so $D$ is not
bipartite. The circuits of $U_{1,4}$ are the $2$-subsets of $E$,
so both $D^*$ and $D^{+*}$ are bipartite. Hence $101$ and $110$
hold, whereas $011$ does not. Since $101+110=011$
in $\Ftwo^3$, the Fano-line implication on
$\{011,101,110\}$ fails. Thus Theorem~\ref{thm:fano-framework}
does not hold for all vf-safe
delta-matroids.
\end{example}

\section*{Acknowledgements}
This work is supported by NSFC (Nos. 12571379, 12471326).

\end{document}